\documentclass[11pt]{amsart}

\usepackage{lmodern} 
\usepackage{microtype} 
\usepackage[english]{babel} 
\usepackage{mathtools} 
\usepackage{hyperref} 
\usepackage[hyperpageref]{backref} 
\usepackage[msc-links]{amsrefs}

\numberwithin{equation}{section} 
\theoremstyle{plain} 
\newtheorem{theorem}{Theorem}[section] 
\newtheorem{lemma}[theorem]{Lemma} 
\newtheorem{corollary}[theorem]{Corollary} 

\theoremstyle{definition} 
\newtheorem{problem}{Problem}

\theoremstyle{remark}
\newtheorem*{remark}{Remark}

\DeclareMathOperator{\mre}{Re} 
\newcommand{\mh}{m_{\mathrm{h}}}

\begin{document} 
\title[Contractive Hardy--Littlewood inequalities]{Contractive Hardy--Littlewood inequalities in the Dirichlet range} 
\date{\today} 

\author[Brevig]{Ole Fredrik Brevig} 
\address{Department of Mathematical Sciences, Norwegian University of Science and Technology (NTNU), 7491 Trondheim, Norway} 
\email{ole.brevig@ntnu.no}

\author[Kulikov]{Aleksei Kulikov} 
\address{University of Copenhagen, Department of Mathematical Sciences, Universitetsparken 5, 2100 Copenhagen, Denmark} 
\email{lyosha.kulikov@mail.ru}

\author[Seip]{Kristian Seip} 
\address{Department of Mathematical Sciences, Norwegian University of Science and Technology (NTNU), 7491 Trondheim, Norway} 
\email{kristian.seip@ntnu.no}

\author[Zlotnikov]{Ilya Zlotnikov} 
\address{Department of Mathematical Sciences, Norwegian University of Science and Technology (NTNU), 7491 Trondheim, Norway} 
\email{ilia.k.zlotnikov@ntnu.no}

\thanks{Ole Fredrik Brevig was supported by Grant 354537 of the Research Council of Norway. Aleksei Kulikov was supported by the VILLUM Centre of Excellence for the Mathematics of Quantum Theory (QMATH) with Grant No.10059. Kristian Seip and Ilya Zlotnikov were supported by Grant 334466 of the Research Council of Norway.}

\maketitle
\begin{abstract}
	The class $A_\alpha^p$ consists of those analytic functions $f$ in the unit disc such that
	\[\|f\|_{\alpha,p}^p \coloneq |f(0)|^p+\int_0^1 \left(\frac{d}{dr} M_p^p(r,f)\right) (1-r^2)^{\alpha-1} \,dr < \infty,\]
	where $M_p^p(r,f)$ is the radial integral mean of $|f|^p$ and $0<\alpha, p <\infty$. For $\alpha>1$, $A_\alpha^p$ is the standard weighted Bergman space, and $A_1^p=H^p$. We consider $A_\alpha^p$ for $0<\alpha<1$ and show that (weighted) isometric conformal invariance extends to this range, and we also clarify the relation between $A_{\alpha}^p$ and the classical Besov spaces. Our main result is the contractive inequality $\|f\|_{\beta,q} \leq \|f\|_{\alpha,p}$, valid when $0<\alpha<\beta<\infty$ and $\alpha/p=\beta/q$. We also identify the functions for which equality is attained. We thus extend recent results of the second-named author ($1\leq \alpha<\beta$) and Llinares ($\beta=1$ and $p=2$). The extension of results from the classical range $1\leq \alpha < \infty$ to the Dirichlet range $0<\alpha <1$ uses arguments relying on analytic continuation. 
\end{abstract}

\section{Introduction} Let $f$ be an analytic function in the unit disc $\mathbb{D}$ in the complex plane. Consider the integral means 
\[M_p^p(r,f) \coloneq \int_0^{2\pi} |f(r e^{i\theta})|^p \,\frac{d\theta}{2\pi}\]
for $0<p<\infty$ and $0<r<1$. An immediate consequence of the well-known Hardy--Stein identity (see \cite{Stein1933} or \eqref{eq:hardystein} below) is that the function $r \mapsto M_p^p(r,f)$ is increasing and differentiable for $0<r<1$. It follows that if $0<p<\infty$ and $0<\alpha<\infty$, then the quantity 
\begin{equation}\label{eq:Apa} 
	\|f\|_{\alpha,p}^p \coloneq |f(0)|^p+\int_0^1 \left(\frac{d}{dr} M_p^p(r,f)\right) (1-r^2)^{\alpha-1} \,dr, 
\end{equation}
is well-defined and nonnegative. Let $A^p_\alpha$ stand for the class of all analytic functions $f$ in $\mathbb{D}$ such that $\|f\|_{\alpha,p} < \infty$. 

Integration by parts shows that if $1<\alpha<\infty$, then 
\begin{equation}\label{eq:bergman} 
	\|f\|_{\alpha,p}^p = (\alpha-1)\int_{\mathbb{D}} |f(z)|^p \,(1-|z|^2)^{\alpha-2} \, dm(z), 
\end{equation}
where $m$ denotes Lebesgue area measure normalized so that $m(\mathbb{D})=1$. This means that, in this range, the classes $A^p_\alpha$ coincide with the standard weighted Bergman spaces. Similarly, $A^p_1$ is the classical Hardy space $H^p$. The purpose of the present paper is to study $A^p_\alpha$ in the Dirichlet range $0<\alpha<1$.

The reasoning behind definition \eqref{eq:Apa} is that the classes $A^p_\alpha$ in the Dirichlet range should inherit certain \emph{geometric} properties of the Hardy and Bergman spaces. Our first result in this direction concerns (weighted) conformal invariance. If $f$ is an analytic function in $\mathbb{D}$, then we set
\[T_{w,\kappa} f(z) \coloneq f\left(\frac{w-z}{1-\overline{w}z}\right) \frac{(1-|w|^2)^\kappa}{(1-\overline{w}{z})^{2\kappa}}\]
for a point $w$ in $\mathbb{D}$ and $\kappa>0$. The following result is well known in the range $1 \leq \alpha < \infty$ (see e.g.~\cite{AM2021}*{Example~1}) and can be deduced from \eqref{eq:bergman} by a change of variables. 
\begin{theorem}\label{thm:mobinv} 
	Fix $0<\alpha<\infty$ and $0<p<\infty$. If $f$ is in $A^p_\alpha$, then so is $T_{w,\alpha/p} f$ for every $w$ in $\mathbb{D}$ and
	\[\|f\|_{\alpha,p} = \|T_{w,\alpha/p}f\|_{\alpha,p}.\]
\end{theorem}

Let us paraphrase Theorem~\ref{thm:mobinv} as saying that the class $A^p_\alpha$ enjoys isometric conformal invariance of index $\kappa=\alpha/p$. We will see that Theorem~\ref{thm:mobinv} follows by a suitable analytic continuation of the identity from the classical range $1\leq \alpha<\infty$ to the Dirichlet range $0<\alpha<1$. Theorem~\ref{thm:mobinv} yields the following sharp pointwise estimate for the elements of $A^p_\alpha$. 
\begin{corollary}\label{cor:pest} 
	Fix $0<\alpha<\infty$ and $0<p<\infty$. The estimate
	\[|f(w)|^p (1-|w|^2)^{\alpha} \leq \|f\|_{\alpha,p}^p,\]
	holds for every $f$ in $A^p_\alpha$. Equality is attained in this bound if and only if $f(z) = C\left(1-\overline{w}z\right)^{-2\alpha/p}$ for a constant $C$ and a point $w$ in $\mathbb{D}$. 
\end{corollary}

Corollary~\ref{cor:pest} follows at once from Theorem~\ref{thm:mobinv} provided we can establish the special case $w=0$. However, this special case is immediate from \eqref{eq:Apa} and the Hardy--Stein identity \eqref{eq:hardystein}.

The main result of the present paper concerns the relationship between the quantities \eqref{eq:Apa} for different classes $A^p_\alpha$ that enjoy conformal invariance of the same index.
\begin{theorem}\label{thm:HL} 
	If $0<\alpha<\beta<\infty$ and $0<p<q<\infty$ satisfy $\alpha/p=\beta/q$, then 
	\begin{equation}\label{ineq:contractive_main} 
		\|f\|_{\beta,q} \leq \|f\|_{\alpha,p} 
	\end{equation}
	holds for every $f$ in $A^p_\alpha$. Equality is attained in this bound if and only if $f(z) = C\left(1-\overline{w}z\right)^{-2\alpha/p}$ for a constant $C$ and a point $w$ in $\mathbb{D}$. 
\end{theorem}

Theorem~\ref{thm:HL} is a contribution to a long line of research that originated in the work of Hardy and Littlewood~\cite{HL1932}. What is important from our point of view is that the inequality is \emph{contractive}, i.e. that the constant in the inequality is $1$. To the best of our knowledge, the first contractive Hardy--Litlewood inequality is due to Carleman (see Vukoti\'c's exposition in \cite{Vukotic2003}) and corresponds to the case $\alpha=1$ and $\beta=2$ in Theorem~\ref{thm:HL}. The fact that Carleman's inequality is contractive is crucial for its application by Helson~\cite{Helson2006} to multiplicative Hankel matrices. More recently, the second-named author~\cite{Kulikov2022} established Theorem~\ref{thm:HL} for $1 \leq \alpha < \infty$, and Llinares~\cite{Llinares2024} did the same for the pair $\beta=1$ and $p=2$ (and all $0<\alpha<1$). This resolved several conjectures from~\cite{BOSZ2018}.

Our main motivation for studying $A^p_\alpha$ in the range $0<\alpha<1$ and establishing Theorem~\ref{thm:HL} was to place the results of \cite{Kulikov2022} and \cite{Llinares2024} in a common framework. To this end, we will rely crucially on the techniques developed in \cite{Kulikov2022} and use also here, though in a less straightforward way than in the proof of Theorem~\ref{thm:mobinv}, analytic continuation to arrive at the extended range for the parameters $\alpha$ and $\beta$. 

Curiously, Theorem~\ref{thm:HL} allows us to obtain the following asymptotic strengthening of Corollary~\ref{cor:pest}. 
\begin{corollary}\label{cor:uniformzero} 
	Fix $0 < \alpha < \infty$ and $0 < p < \infty$. If $f$ is in $A^p_\alpha$, then 
	\begin{equation}\label{eq:bound_lim_zero} 
		\lim_{|w|\to 1^-} |f(w)|^p(1-|w|^2)^\alpha = 0. 
	\end{equation}
\end{corollary}

Theorem~\ref{thm:HL} yields another companion to the pointwise estimate from Corollary~\ref{cor:pest}. Indeed, we obtain a dichotomy for the classical majorant function (see~\cites{Boas2000,Bohr1914}) which for $f(z) = \sum_{j\geq0} a_j z^j$ is defined as
\[Mf(z) \coloneq \sum_{j=0}^\infty |a_j| z^j.\]
Note that, trivially, $|f(z)| \leq Mf(|z|)$, which justifies the name. 
\begin{corollary}\label{cor:majorant} 
	Fix $0 < \alpha < \infty$. 
	\begin{enumerate}
		\item[(a)] If $0 < p \leq 2$, then the estimate
		\[\left(Mf(|w|)\right)^p (1-|w|^2)^{\alpha} \leq \|f\|_{\alpha,p}^p,\]
		holds for every $f$ in $A^p_\alpha$ and every $w$ in $\mathbb{D}$. 
		\item[(b)] If $2<p<\infty$ and $w\neq0$, then there is a function $f$ in $A^p_\alpha$ such that
		\[\left(Mf(|w|)\right)^p (1-|w|^2)^{\alpha} > \|f\|_{\alpha,p}^p.\]
	\end{enumerate}
\end{corollary}

As indicated above, definition \eqref{eq:Apa} purports to extend the Bergman range beyond the (perhaps) natural Hardy endpoint to the Dirichlet range. Another way to achieve the same goal is via Besov spaces. For simplicity we restrict our attention to the case $p+\alpha>1$ and let $B^p_\alpha$ denote the space of analytic functions in $\mathbb{D}$ such that $\|f\|_{B^p_\alpha}<\infty$, where
\[\|f\|_{B^p_\alpha}^p = |f(0)|^p + (\alpha+p-1)\int_{\mathbb{D}} |f'(z)|^p (1-|z|^2)^{\alpha+p-2}\,dm(z).\]
The Besov spaces $B^p_\alpha$ are well known to enjoy conformal invariance of index $\alpha/p$ (see e.g. \cite{AM2021}*{Example 3}). In the Bergman range $1<\alpha<\infty$, we have $A^p_\alpha = B^p_\alpha$ as sets and with equivalence of norms (see \cite{PR2021}*{Theorem~5} for a more general result). For $0 < \alpha \leq 1$, the situation is as follows.
\begin{theorem}\label{thm:besovcomp} 
	Suppose that $0<\alpha \leq 1$. If 
	\begin{enumerate}
		\item[(a)] $0<p \leq 2$, then $B^p_\alpha \subset A^p_\alpha$, 
		\item[(b)] $2 \leq p < \infty$, then $A^p_\alpha \subset B^p_\alpha$. 
	\end{enumerate}
	Moreover, $A^p_\alpha \neq B^p_\alpha$ unless $p=2$. 
\end{theorem}

Theorem~\ref{thm:besovcomp} is not new for $\alpha=1$: In this case, (a) goes back at least to Vinogradov~\cite{Vinogradov1995}, while (b) is a classical inequality due to Littlewood and Paley~\cite{LP1936}. The final assertion is also well known for $\alpha=1$ and can be established with the help of suitable lacunary series (see e.g.~\cite{BGP2004}*{p.~840}). The proof of (a) and (b) of Theorem~\ref{thm:besovcomp} rests on well known techniques for Besov spaces, more precisely on work by Luecking \cite{Luecking88} in the range $p>2$ and by Dyakonov \cite{Dyakonov98} in the range $1\leq p <2$. Our example functions showing that the inclusions are strict arise from analytic functions that, away from their zeros, grow like a negative power of the distance to the boundary. 

Aleman and Mas \cite{AM2021}*{Section~4} have determined the largest and smallest Banach spaces of analytic functions in $\mathbb{D}$ that enjoy conformal invariance of a fixed index $0<\kappa<\infty$. They prove that the smallest such space is $B^1_\kappa$ and that the largest is a Korenblum growth class. They also establish that the Besov spaces $B^p_{\kappa/p}$ can be obtained by (complex) interpolation between the largest and smallest spaces. As noted in \cite{AM2021}*{p.~4}, the Hardy spaces are missing from this interpolation chain. This is in line with Theorem~\ref{thm:besovcomp}.

The reader may have noticed that we consistently refer to the formula \eqref{eq:Apa} as a quantity and to $A^p_\alpha$ as a class. If $1 \leq \alpha < \infty$, then \eqref{eq:Apa} is a (quasi-)norm and $A^p_\alpha$ is a linear space. It would be interesting to know if this property extends to the Dirichlet range.
\begin{problem}\label{prob:linear} 
	Fix $0<\alpha<1$ and $p \neq 2$. Is $A^p_\alpha$ a linear space? 
\end{problem}

The case $p=2$ has to be excluded from Problem~\ref{prob:linear}, since $A_\alpha^2$ is easily verified to be a Hilbert space (see Theorem~\ref{thm:hilbertspace} below). Aleman and Mas~\cite{AM2021}*{Theorem~5} have in fact established that $A^2_\alpha$ is the \emph{unique} Hilbert space (up to equivalence of norms) that enjoys conformal invariance of index $\kappa=\alpha/2$. Our choice of norm is canonical in the sense that it is the only norm such that the maps $T_{w,\alpha/2}$ are isometries on $A^2_\alpha$. The Hilbert space structure of $A^2_\alpha$ plays a role in the proof of Corollary~\ref{cor:majorant}.

We will now add two more results which, in spite of the unsettled Problem~\ref{prob:linear}, reveal that $A_\alpha^p$ have some desirable properties in the range $0<\alpha<1$, complementing in a natural way well-known results for $\alpha\geq 1$. We consider first the shift operator $S$ which acts on analytic functions $f$ as
\[Sf(z)\coloneqq zf(z).\]
To state our result regarding $S$, we introduce the following terminology. We say that $T \colon A_\alpha^p\to A_{\alpha}^p$ is a strict contraction (respectively a strict expansion) on $A_{\alpha}^p$ if $\| Tf\|_{\alpha,p}<\|f\|_{\alpha,p}$ (respectively $\| Tf\|_{\alpha,p}>\|f\|_{\alpha,p}$), in either case on the proviso that $f\not\equiv 0$. We say that $T$ is norm attaining on $A_{\alpha}^p$ if there exists a function $f_0$ in $A_\alpha^p$ with $\|f_0\|_{\alpha,p}=1$ such that $\| Tf_0 \|_{\alpha,p} =\| T \|_{\alpha,p}$, where
\[\| T \|_{\alpha,p} \coloneqq \sup_{\substack{f\in A_\alpha^p \\ f \not \equiv 0}} \frac{\| Tf \|_{\alpha,p}}{\| f\|_{\alpha,p}}.\]
It is a trivial fact that $S$ is a strict contraction on $A_{\alpha}^p$ which fails to be norm attaining when $\alpha>1$, and also that $S$ is an isometry on $H^p=A_{1}^p$. In the range $0<\alpha<1$, we have the following.
\begin{theorem}\label{thm:exp} 
	Fix $0<\alpha<1$ and $0<p<\infty$. The shift operator $S$ is a strict expansion on $A_{\alpha}^p$ that is norm attaining with
	\[\|S\|_{\alpha,p}^p = 1+2(1-\alpha) \int_0^1 (1-r^p) (1-r^2)^{\alpha-2} \,rdr.\]
	Moreover, $\|S f_0\|_{\alpha,p}= \|S\|_{\alpha,p}$ if and only if $f_0 \equiv C$. 
\end{theorem}
Since $\alpha \mapsto \|f\|_{\alpha,p}$ is increasing in $\alpha$, it follows that if $f$ is in $A^p_\alpha$ for $0<\alpha<1$, then $f$ is in $H^p$ and, in particular, that $f$ admits an inner-outer factorization. The proof of Theorem~\ref{thm:exp} can be elaborated to yield the following general result about division by inner functions. 
\begin{theorem}\label{thm:inner} 
	Suppose that $0<\alpha<1$ and $0<p<\infty$, and let $f$ be a nontrivial function in $A_{\alpha}^p$. If $I$ is a nontrivial inner function dividing $f$, then
	\[ \| f/I \|_{\alpha,p} < \| f\|_{\alpha,p}. \]
\end{theorem}
An immediate corollary of this theorem is that the outer part of a function in $A_{\alpha}^p$ must itself belong to $A_{\alpha}^p$ when $0<\alpha<1$. This is a property that $A_{\alpha}^p$ shares with $B_{\alpha}^p$ (see Dyakonov's remark in \cite{Dyakonov98}*{p.~144}).

\subsection*{Organization} This paper consists of five sections. Section~\ref{sec:preliminaries} contains some basic properties of $A^p_\alpha$ and culminates with the proof of Theorem~\ref{thm:mobinv}. The proof of Theorem~\ref{thm:HL} and its two corollaries can be found in Section~\ref{sec:HL}. Section~\ref{sec:besov} is devoted to the comparison of $A^p_\alpha$ and $B^p_\alpha$ and contains the proof of Theorem~\ref{thm:besovcomp}. The final Section~\ref{sec:shift} contains the proofs of Theorem~\ref{thm:exp} and Theorem~\ref{thm:inner}.

\section{Preliminaries} \label{sec:preliminaries} 
The Hardy--Stein identity \cite{Stein1933} for an analytic function $f$ in the unit disc is 
\begin{equation}\label{eq:hardystein} 
	\frac{d}{dr} M_p^p(r,f) = \frac{p^2}{2r}\int_{r \mathbb{D}} |f(z)|^{p-2} |f'(z)|^2 \,dm(z). 
\end{equation}
As mentioned above, it follows readily from \eqref{eq:hardystein} that the function $r \mapsto M_p^p(r,f)$ is continuously differentiable and strictly increasing (unless $f$ is identically equal to a constant). Using \eqref{eq:hardystein} and Fubini's theorem, we can rewrite \eqref{eq:Apa} as 
\begin{equation}\label{eq:littlewoodpaley} 
	\|f\|_{\alpha,p}^p = |f(0)|^p + \frac{p^2}{4}\int_{\mathbb{D}} |f(z)|^{p-2} |f'(z)|^2 \omega_\alpha(|z|^2)\,dm(z), 
\end{equation}
where
\[\omega_\alpha(x) = \int_x^1 (1-r)^{\alpha-1} \,\frac{dr}{r}.\]
If $\alpha=1$, this integral can be computed explicitly, and we obtain the classical Littlewood--Paley formula for the $H^p$ norm. This indicates that \eqref{eq:littlewoodpaley} is perhaps the most appropriate way of expressing the formula \eqref{eq:Apa} as a Littlewood--Paley integral. For $\alpha\neq1$, the functional equation 
\[\omega_\alpha(x) = \frac{(1-x)^\alpha}{\alpha} + \omega_{\alpha+1}(x) \]
allows us to obtain the following less precise version of \eqref{eq:littlewoodpaley}, that will find use in the proof of Theorem~\ref{thm:besovcomp}.
\begin{lemma}\label{lem:LP} 
	Fix $0<\alpha<\infty$ and $0<p<\infty$. We have
	\[\|f\|_{A^p_\alpha}^p \asymp |f(0)|^p + \int_{\mathbb{D}} |f(z)|^{p-2} |f'(z)|^2 (1-|z|^2)^\alpha \,dm(z).\]
\end{lemma}

Let us continue with a few observations on $A^2_\alpha$. To facilitate this, we recall the binomial series 
\begin{equation}\label{eq:binomialseries} 
	\frac{1}{(1-z)^\alpha} = \sum_{k=0}^\infty c_\alpha(k) z^k, \qquad\text{where}\qquad c_\alpha(k) = \binom{\alpha+k-1}{k} 
\end{equation}
for $0<\alpha<\infty$. For $k=1,2,3,\ldots,$ the identity 
\begin{equation}\label{eq:betaint} 
	\frac{1}{c_\alpha(k)} = 2 \int_0^1 r^{2k-1} (1-r^2)^{\alpha-1}\,dr 
\end{equation}
can be deduced from the properties of the Beta function. 
\begin{lemma}\label{thm:hilbertspace} 
	Suppose that $0<\alpha<\infty$. If $f(z) = \sum_{k\geq0} a_k z^k$, then
	\[\|f\|_{A^2_\alpha}^2 = \sum_{k=0}^\infty \frac{|a_k|^2}{c_\alpha(k)}\]
	and, consequently, $A^2_\alpha$ is a Hilbert space. 
\end{lemma}
\begin{proof}
	By orthogonality, we have
	\[\frac{d}{dr} M_2^2(r,f) = \sum_{k=1}^\infty |a_k|^2 (2k) r^{2k-1}.\]
	The asserted result now follows from \eqref{eq:Apa} and \eqref{eq:betaint}. 
\end{proof}

We now turn to the proof of Theorem~\ref{thm:mobinv}, which relies on a preliminary result that will also find use in the proof of Theorem~\ref{thm:HL}. To state it, we set $f_\varrho(z) \coloneq f(\varrho z)$ for $0<\varrho \leq 1$ and $f$ analytic in $\mathbb{D}$.
\begin{lemma}\label{lem:rho1} 
	Fix $0<\alpha<\infty$ and $0<p<\infty$. If $f$ is analytic in $\mathbb{D}$, then the function $\varrho \mapsto \|f_\varrho\|_{\alpha,p}$ is increasing for $0<\varrho<1$. Moreover,
	\[\lim_{\varrho \to 1^-} \|f_\varrho\|_{\alpha,p} = \|f\|_{\alpha,p}.\]
\end{lemma}
\begin{proof}
	It is clear that $f_\varrho$ is in $A^p_\alpha$ for every $0<\varrho<1$. To establish the first assertion, we note that $M_p^p(r, f_\varrho) = M_p^p (\varrho r, f)$ so that 
	\begin{equation}\label{eq:fvarrho} 
		\|f_\varrho\|_{\alpha,p}^p = |f(0)|^p+\int_0^1 \left(\frac{d}{dr} M_p^p(\varrho r, f)\right) (1-r^2)^{\alpha-1} \,dr. 
	\end{equation}
	Hardy's convexity theorem \cite{Hardy1915} asserts that $M_p(r,f)$ is a logarithmically convex function of $\log{r}$. It follows that the function
	\[x \mapsto \frac{d}{dx} M_p^p(e^x,f)\]
	is increasing on $-\infty<x<0$, so that
	\[\frac{d}{dx} M_p^p(e^x,f) \leq \frac{d}{dy} M_p^p(e^y,f)\]
	for $x = \log{(\varrho_1 r)}$ and $y = \log{(\varrho_2 r)}$ whenever $0<\varrho_1 \leq \varrho_2 \leq 1$ and $0<r<1$. Hence 
	\begin{equation}\label{eq:multincreasing} 
		\frac{d}{dr} M_p^p(\varrho_1 r,f) = \frac{1}{r} \frac{d}{dx} M_p^p(e^x,f) \leq \frac{1}{r}\frac{d}{dy} M_p^p(e^y,f) = \frac{d}{dr} M_p^p(\varrho_2 r,f). 
	\end{equation}
	We insert this estimate into \eqref{eq:fvarrho} to see that $\|f_{\varrho_1}\|_{\alpha,p} \leq \|f_{\varrho_2}\|_{\alpha,p}$, which completes the proof of the first assertion. The second assertion follows from \eqref{eq:fvarrho}, \eqref{eq:multincreasing}, and the monotone convergence theorem. 
\end{proof}
\begin{proof}
	[Proof of Theorem~\ref{thm:mobinv}] If $g$ is any analytic function in the unit disc, then
	\[\lim_{\varrho \to 1^-} \frac{d}{dr} M_p^p(r,g_\varrho) = \frac{d}{dr} M_p^p(r,g)\]
	for every fixed $0<r<1$. It follows from this and Fatou's lemma that
	\[\|T_{w,\alpha/p} f\|_{\alpha,p}^p \leq \liminf_{\varrho \to 1^-} \|T_{w,\alpha/p} f_\varrho \|_{\alpha,p}^p \leq \sup_{0<\varrho<1} \|T_{w,\alpha/p} f_\varrho \|_{\alpha,p}^p. \]
	Let us now consider a fixed $0<\varrho<1$. We see from formula \eqref{eq:Apa} that
	\[F_\varrho(\alpha) \coloneq \| T_{w,\alpha/p} f_\varrho \|_{\alpha,p}^p\]
	extends to an analytic function in the right half-plane $\mre{\alpha}>0$, since $f_\varrho$ is analytic in the disc $\varrho^{-1} \mathbb{D}$. By the classical weighted conformal invariance of $A^p_\alpha$ for real $\alpha > 1$ we have that $F_\varrho(\alpha)=\|f_\varrho\|_{\alpha,p}^p$ for real $\alpha>1$. By the identity theorem for analytic functions, this holds true for all $\alpha$ in the right half-plane. In particular,
	\[\sup_{0<\varrho<1} \|T_{w,\alpha/p} f_\varrho \|_{\alpha,p}^p = \sup_{0<\varrho<1} \| f_\varrho \|_{\alpha,p}^p = \|f\|_{\alpha,p}^p,\]
	where the final equality is Lemma~\ref{lem:rho1}. This shows that $\| T_{w,\alpha/p} f\|_{\alpha,p} \leq \|f\|_{\alpha,p}$, so $T_{w,\alpha/p} f$ is in $A^p_\alpha$. Since $f = T_{w,\alpha/p} T_{w,\alpha/p} f$, the same argument now gives that $\|f\|_{\alpha,p} \leq \|T_{w,\alpha/p} f\|_{\alpha,p}$. 
\end{proof}

The analytic continuation argument used in the proof of Theorem~\ref{thm:mobinv} can be used to establish what is sometimes called the \emph{power trick} in our setting. 
\begin{theorem}\label{thm:powertrick} 
	Fix $0<p<\infty$ and $0<\alpha<\infty$. If $n$ is a positive integer, then
	\[\|f^n\|_{\alpha,p}^p = \|f\|_{\alpha,pn}^{pn}.\]
\end{theorem}
\begin{proof}
	If $1<\alpha<\infty$, this is trivial due to \eqref{eq:bergman}. Since $(f^n)_\varrho = (f_\varrho)^n$, we complete the proof using analytic continuation and Lemma~\ref{lem:rho1} as above. 
\end{proof}

Theorem~\ref{thm:powertrick} illustrates again that certain properties of the quantities \eqref{eq:Apa} extend from the Bergman range $1<\alpha<\infty$ to the Dirichlet range $0<\alpha<1$. As a companion to Theorem~\ref{thm:powertrick}, we offer the following.
\begin{problem}
	Fix $0<\alpha<1$ and let $f$ be an analytic function in $\mathbb{D}$. Is the function $p \mapsto \|f\|_{A^p_\alpha}$ increasing? 
\end{problem}

\section{Contractive Hardy--Littlewood inequalities} \label{sec:HL}

Let us begin by recalling the basic setup from \cite{Kulikov2022}. The hyperbolic measure on $\mathbb{D}$ is defined by
\[d\mh(z) \coloneq \frac{dm(z)}{(1-|z|^2)^{2}}.\]
For a fixed function $f$ and $0<\sigma<\infty$, we define 
\[\mu(t) \coloneq \mh(\{z\,:\, |f(z)|^{\sigma}(1-|z|^2) > t\}). \]
We will rely on the following result, which is contained in \cite{Kulikov2022}*{Theorem~2.1}.
\begin{lemma}\label{lem:kulikovg} 
	Fix $0<\sigma<\infty$ and let $f$ be an analytic function in $\mathbb{D}$ such that the function $u(z) = |f(z)|^\sigma (1-|z|^2)$ is bounded in the unit disc and tends uniformly to $0$ as $|z|\to 1^-$. The function
	\[g(t) \coloneqq t(\mu(t)+1)\]
	is non-increasing on $(0,t_0)$ where $t_0 = \max_{z \in \mathbb{D}} u(z)$. 
\end{lemma}

Note that in the present section $\sigma = \kappa^{-1}$ compared to the Introduction. 

If $f$ satisfies the assumptions of Lemma~\ref{lem:kulikovg}, then we set 
\begin{equation}\label{eq:Phi_def} 
	\Phi(\alpha,\sigma,f) \coloneqq t_0^{\alpha} - \alpha \int_0^{t_0} g'(t) t^{\alpha - 1} \, dt 
\end{equation}
for $0<\alpha ,\sigma<\infty$ and for $g$ as in Lemma~\ref{lem:kulikovg}.

We divide the proof of Theorem~\ref{thm:HL} into five steps, and our plan is as follows. Note that the first two steps are logically independent of each other, but that both rely on Lemma~\ref{lem:kulikovg}.
\begin{itemize}
	\item[\textbf{1.}] Prove that $\Phi(\alpha,\sigma,f) = \|f\|^{\sigma\alpha}_{\alpha,\sigma\alpha}$ under the additional assumptions that $f$ is analytic in the closed unit disc (in particular: Lemma~\ref{lem:kulikovg} applies) and that $f$ does not vanish on the unit circle. This is the part of the proof of Theorem~\ref{thm:HL} that relies on analytic continuation. This result is Lemma~\ref{lem:alternative_norm} below. 
	\item[\textbf{2.}] Establish the analog of inequality~\eqref{ineq:contractive_main} of Theorem~\ref{thm:HL} if $f$ satisfies the assumptions of Lemma~\ref{lem:kulikovg} and if $\|f\|_{\alpha,p}^p$ and $\|f\|_{\beta,q}^q$ are replaced by $\Phi(\alpha,p/\alpha,f)$ and $\Phi(\beta,q/\beta,f)$ in \eqref{ineq:contractive_main}. Moreover, we describe in terms of $\Phi$ (see~\eqref{eq:Phi_alphabeta}) the functions for which the equality in~\eqref{ineq:contractive_main} is attained. This is Lemma~\ref{lem:contr_bd}. 
	\item[\textbf{3.}] Employ Lemma~\ref{lem:rho1} to establish the inequality \eqref{ineq:contractive_main} of Theorem~\ref{thm:HL} without the additional assumptions on $f$. The key point is that if $f$ is a nontrivial analytic function in $\mathbb{D}$, the functions $f_\varrho(z) = f(\varrho z)$ satisfy the assumptions of Lemma~\ref{lem:alternative_norm} for almost every $0<\varrho<1$ and Lemma~\ref{lem:contr_bd} for every $0<\varrho<1$. This is Lemma~\ref{lem:prototype}. 
\end{itemize}
In order to complete the proof of Theorem~\ref{thm:HL} we need to prove the final assertion. 
\begin{itemize}
	\item[\textbf{4.}] Prove Corollary~\ref{cor:uniformzero} (using only Lemma~\ref{lem:prototype}). The virtue of having Corollary~\ref{cor:uniformzero} in this setting is that we may now apply Lemma~\ref{lem:contr_bd} to any function in $A^p_\alpha$. This puts us in a position to finish the proof of Theorem~\ref{thm:HL}. 
	\item[\textbf{5.}] The basic idea is to show that there is some $\gamma_0$ between $\alpha$ and $\beta$ such that the quantities $\|\cdot\|_{\beta, \sigma \beta}$ and $\|\cdot\|_{\gamma_0, \sigma \gamma_0}$ can be represented in terms of $\Phi$. Since the functions $f$ for which we have $\|f\|_{\alpha, \sigma \alpha} = \|f\|_{\beta, \sigma \beta} $ should also satisfy $\|f\|_{\gamma_0, \sigma \gamma_0} = \|f\|_{\beta, \sigma \beta} $, we can apply Lemma~\ref{lem:contr_bd} to deduce that $\Phi(\beta,\sigma, f)$ satisfies~\eqref{eq:Phi_alphabeta}. Using this and representation of $\|\cdot\|_{\beta, \sigma \beta}$ in terms of $\Phi$, we apply Corollary~\ref{cor:pest} to obtain the desired conclusion. 
\end{itemize}

As outlined above, we begin by proving that \eqref{eq:Phi_def} provides an alternative expression for the quantity \eqref{eq:Apa} under certain additional assumptions on $f$.
\begin{lemma}\label{lem:alternative_norm} 
	Let $f$ be an analytic function in the closed unit disc that does not vanish on the unit circle. If $0<\alpha,\sigma<\infty$, then 
	\begin{equation}\label{eq:norm_alt} 
		\Phi(\alpha,\sigma,f) = \|f\|^{\sigma\alpha}_{\alpha,\sigma\alpha}. 
	\end{equation}
\end{lemma}
\begin{proof}
	Since $f$ is assumed to be analytic in the closed unit disc, it is plainly bounded there. Let us assume for a moment that $ \alpha > 1$ and use \eqref{eq:bergman} to compute
	\[\|f\|^{\sigma\alpha}_{\alpha,\sigma\alpha} = (\alpha-1)\int_{0}^{t_0^{\alpha}} \mu(t^{1/\alpha}) \, dt = (\alpha-1) \int_{0}^{t_0^{\alpha}} \frac{g(t^{1/\alpha})}{t^{1/\alpha}} \,dt -(\alpha-1)t_0^{\alpha}.\]
	After integration by parts, we get
	\[\|f\|^{\sigma\alpha}_{\alpha,\sigma\alpha} = -(\alpha-1) t_0^{\alpha} + \alpha g(t_0) t_0^{\alpha-1} - \alpha\lim_{t \to 0^+} g(t^{1/\alpha}) t^{1-{1/\alpha}} - \int_0^{t_0^{\alpha}} g'(t^{1/\alpha}) \, dt.\]
	Note that $\mu(t_0) = 0$ whence $g(t_0) = t_0$, and since $g$ is bounded (because $f$ is bounded) and $\alpha> 1$, we get
	\[\lim_{t \to 0^+} g(t^{1/\alpha}) t^{1-{1/\alpha}} = 0.\]
	Therefore, after a change of variables in the last integral, we obtain~\eqref{eq:norm_alt} for $\alpha > 1$.
	
	To extend \eqref{eq:norm_alt} to $\alpha > 0$, we will show that the functions on both sides of \eqref{eq:norm_alt} admit analytic continuation for all $\alpha$ in the right half-plane. To this end, we begin by verifying that either expression is well defined for $\alpha$ in this range. It is immediate from \eqref{eq:Apa} and the assumption that $f$ is analytic in the closed unit disc that 
	\[\|f\|_{\alpha, \sigma \alpha} < \infty. \]
	for all $\alpha>0$. We show next that the left-hand side of \eqref{eq:norm_alt} is also finite for $\alpha>0$. Since $f$ is bounded, we have that $|f(z)|^\sigma (1-|z|^2) \to 0$ uniformly as $|z|\to 1^-$, so it follows that for any fixed $c>0$, we have
	\[-\int_{c}^{t_0} g'(t) t^{\alpha-1} \, dt \leq c^{\alpha-1} (g(c) - g(t_0)) < \infty.\]
	It therefore suffices to verify that for some fixed $c>0$ we have 
	\begin{equation}\label{eq:fin_cont} 
		- \int_0^{c} g'(t) t^{\alpha - 1} \, dt < \infty. 
	\end{equation}
	for $0<\alpha<1$. It follows from Lemma~\ref{lem:kulikovg} that $-g' \geq 0$, so we can use a dyadic decomposition to estimate this integral from above by
	\[\sum_{n=0}^\infty (g(c2^{-n-1}) - g(c2^{-n})) c^{\alpha-1}2^{-n(\alpha-1)}.\]
	To prove that this sum is finite, it suffices to show that there exists $C>0$ such that for every sufficiently small $t > 0$ we have 
	\[g(0) - g(t) \leq C t. \]
	By \cite{Kulikov2022}*{Remark~3.2}, we have $g(0) = \|f\|^{\sigma}_{1,\sigma} = \|f\|_{H^\sigma}^\sigma.$ Therefore, the latter inequality will follow if we can establish that 
	\begin{equation}\label{eq:m_h_below} 
		\mu(t) \geq \frac{\|f\|^{\sigma}_{H^{\sigma}}}{t} - M 
	\end{equation}
	for some positive $M$ independent of $t$. Since $f$ is analytic in the closed unit disc and has no zeros on the unit circle by assumption, there exists $L = L(\sigma,f)>0$ such that
	\[|f(r e^{i\theta})|^{\sigma} \geq |f(e^{i \theta})|^{\sigma} - L(1-r) \geq |f(e^{i \theta})|^{\sigma} - L(1-r^2),\]
	for $0 < r < 1$. It follows that the left-hand side of~\eqref{eq:m_h_below} may be estimated from below by
	\[\mh(\{z = re^{i\theta}\,:\, |f(e^{i \theta})|^{\sigma}(1-r^2) - L(1-r^2)^{2} > t\}).\]
	Set $\varrho \coloneqq 1 - r^2,$ and note that $|f(e^{i \theta})|^{\sigma}\varrho - L\varrho^{2} > t$ is satisfied if and only if $\varrho$ belongs to the segment
	\[I_t:=\left[ \frac{|f(e^{i \theta})|^{\sigma} - \sqrt{|f(e^{i \theta})|^{2\sigma} - 4tL}}{2L}, \frac{|f(e^{i \theta})|^{\sigma} + \sqrt{|f(e^{i \theta})|^{2\sigma} - 4tL}}{2L} \right].\]
	Since $t$ can be chosen sufficiently small and $f$ has no zeros on the unit circle, this set is well defined. Using the substitution $\varrho = 1-r^2$, we next write
	\[\mh(\{z = re^{i\theta}\,:\, |f(e^{i \theta})|^{\sigma}\varrho - L\varrho^{2} > t\}) = \int_0^{2\pi} \int_{I_t} \frac{1}{\varrho^2} \, d \varrho \, \frac{d \theta}{2\pi}.\]
	We proceed with computing
	\[\int_{I_t} \frac{1}{\varrho^2} \, d \varrho = \frac{\sqrt{|f(e^{i\theta})|^{2\sigma} - 4tL}}{t}\geq \frac{|f(e^{i\theta})|^\sigma}{t}-CL, \]
	which yields the bound
	\[\mh(\{z = re^{i\theta}\,:\, |f(e^{i \theta})|^{\sigma}\varrho - L\varrho^{2} > t\}) \geq \frac{\|f\|^{\sigma}_{H^{\sigma}}}{t} -CL\]
	where $C$ is independent of $t$. This finishes the proof of~\eqref{eq:m_h_below}, and therefore~\eqref{eq:fin_cont} has been verified. 

	The next step is to establish that both sides of \eqref{eq:norm_alt} are differentiable with respect to $\alpha.$ We start with the expression on the left-hand side. Differentiating under the integral sign, it is sufficient to show the convergence of the integral 
	\[\int_0^{t_0} -g'(t) t^{\gamma-1}|\log t|\,dt\]
	for all $\gamma > 0$. Taking any $0 < \beta < \gamma$ and using the estimate $|\log t|\leq C_{t_0,\beta,\gamma}t^{\beta-\gamma}$ for $0 < t < t_0$ and the fact that $\Phi(\beta, \sigma, f)$ is finite we get the result.
	
	For the expression $\|f\|_{\gamma,\sigma \gamma}^{\sigma \gamma} $ on the right-hand side of \eqref{eq:norm_alt}, we use the Hardy--Stein identity \eqref{eq:hardystein}. It suffices to show that 
	\begin{equation}\label{eq:surface} 
		\int_{0}^1 \int_{0}^r \int_{0}^{2\pi} \frac{\varrho |f'(\varrho e^{i \theta})|^2 }{r(1-r^2)|f(\varrho e^{i \theta})|^2} ( |f(\varrho e^{it})|^{\sigma} (1-r^2))^{\gamma} \, d \theta \, d\varrho \, dr
	\end{equation}
	is differentiable with respect to $\gamma$ for $\mre{\gamma} > 0$. For a fixed value $\tau$, we now define $D_{\tau} \coloneqq \{(\theta,\varrho, r)\,:\,|f(\varrho e^{it})|^{\sigma}(1-r^2) = \tau\}$. Rewriting \eqref{eq:surface} in terms of integrals over these level surfaces and noting that $|f(\varrho e^{it})|^{\sigma}(1-r^2)\leq t_0$, we see that it remains to show that the function
	\[F(\gamma) \coloneqq \int_{0}^{t_0} \tau^{\gamma} \,d h(\tau) \,\]
	is differentiable for $\mre{\gamma}>0$, where $dh$ is some nonnegative measure. This is done in the same way as in the case of $\Phi(\gamma,\sigma,f)$ by taking some $0 < \beta < \mre{\gamma}$ and using the estimate $|\log t|\leq C_{t_0,\beta,\gamma}t^{\beta-\mre{\gamma}}$ for $0 < t < t_0$.
	
	Since both sides of \eqref{eq:norm_alt} admit analytic continuation to the right half-plane and are equal for real $\alpha > 1$, we obtain that \eqref{eq:norm_alt} holds for all real $\alpha>0$. 
\end{proof}

\begin{remark}
	By being more careful it is in fact possible to show that for the function satisfying the assumptions of Lemma \ref{lem:alternative_norm} we have 
	\[\lim_{t\to 0^+} \mu(t) - \frac{\|f\|^\sigma_{H^\sigma}}{t} = -1 - \frac{\sigma}{4\pi}\int_0^{2\pi}\mre\left(\frac{f'(e^{i\theta})e^{i\theta}}{f(e^{i\theta})}\right)d\theta = -1-\frac{\sigma n}{2},\]
	where $n$ is the number of zeroes of $f$ in the unit disk with multiplicity. Curiously, the expression $\frac{\sigma n}{2}$ is the integral of the distributional Laplacian of $\log |f(z)|^\sigma$ against the measure $dm(z)$. Note that the log-subharmonicity of $\log |f(z)|^\sigma$ played the crucial role in the proof of \cite{Kulikov2022}*{Theorem~2.1}, although we do not know of an a priori reason for its appearance here. 
\end{remark}

We next establish a version of Theorem~\ref{thm:HL} for $\Phi$, provided the assumptions of Lemma~\ref{lem:kulikovg} are met. 
\begin{lemma}\label{lem:contr_bd} 
	Suppose that $0<\alpha<\beta<\infty$ and $0<p<q<\infty$ satisfy $\sigma = p/\alpha=q/\beta$. If $f$ is an analytic function in the unit disc such that
	\[|f(z)|^\sigma (1-|z|^2) \to 0\]
	uniformly as $|z| \to 1^-$, then 
	\begin{equation}\label{eq:fr_HL} 
		\left(\Phi(\beta,\sigma,f)\right)^{1/q} \leq \left(\Phi(\alpha,\sigma,f)\right)^{1/p}. 
	\end{equation}
	Moreover, the equality in \eqref{eq:fr_HL} is attained if and only if either both sides are infinite or 
	\begin{equation}\label{eq:Phi_alphabeta} 
		\Phi(\beta, \sigma, f) = t_0^\beta. 
	\end{equation}
\end{lemma}
\begin{proof}
	Using the definition of $\Phi$ from \eqref{eq:Phi_def}, we rewrite desired the inequality~\eqref{eq:fr_HL} as
	\[\left(t_0^{\alpha} + \alpha \int_0^{t_0} (-g'(t)) t^{\alpha-1} \, dt \right)^{1/(\sigma \alpha)} \geq \left(t_0^{\beta} + \beta \int_0^{t_0} (-g'(t)) t^{\beta-1} \, dt \right)^{1/(\sigma \beta)},\]
	where $t_0$ and $g$ are defined for the function $f$ as in Lemma~\ref{lem:kulikovg}. Note in particular that $t_0$ and $g$ only depend on $f$ and $\sigma$ (and not on $\alpha,\beta,p,q$). Set
	\[J:= \int_0^{t_0} (-g'(t)) \left(\frac{t}{t_0}\right)^{\alpha-1} \, dt\]
	Which we assume is finite. Since $f$ satisfies the assumption of Lemma~\ref{lem:kulikovg}, we infer that $g'\leq 0$. Using this and that $\alpha<\beta$, we get
	\[J \geq \int_0^{t_0} (-g'(t)) \left(\frac{t}{t_0}\right)^{\beta-1} \, dt.\]
	Hence, it suffices to show that
	\[(t_0^{\alpha} + \alpha t_0^{\alpha-1} J)^{1/(\sigma\alpha)} \geq (t_0^\beta + \beta t_0^{\beta-1} J)^{1/(\sigma \beta)}.\]
	Raising both sides to the power $\sigma$ and dividing them by $t_0 > 0$, we get
	\[\left(1 + \alpha \frac{J}{t_0}\right)^{1/\alpha} \geq \left(1 + \beta \frac{J}{t_0}\right)^{1/\beta}.\]
	The latter inequality is true since the function $(1+cx)^{1/x}$ is non-increasing in $x$ for $c \geq 0.$ Moreover, the equality is attained if and only if $c=0$, i.e. $g' \equiv 0$ implying~\eqref{eq:Phi_alphabeta}. 
\end{proof}

We can now prove the first part of Theorem~\ref{thm:HL}.
\begin{lemma}\label{lem:prototype} 
	Suppose that $0<\alpha<\beta<\infty$ and $0<p<q<\infty$ satisfy $p/\alpha=q/\beta$. If $f$ is in $A^p_\alpha$, then
	\[\|f\|_{\beta,q} \leq \|f\|_{\alpha,p}.\]
\end{lemma}
\begin{proof}
	If $f \equiv 0$, there is nothing to do. If $f \not \equiv 0$, then the function $f_\varrho$ satisfies the assumptions of Lemma~\ref{lem:alternative_norm} for almost every $0<\varrho<1$ and Lemma~\ref{lem:contr_bd} for every $0<\varrho<1$. We obtain the stated result after combining these two results and appealing to Lemma~\ref{lem:rho1}. 
\end{proof}

We now use~Lemma~\ref{lem:prototype} to establish Corollary~\ref{cor:uniformzero}.
\begin{proof}
	[Proof of Corollary~\ref{cor:uniformzero}] For $1 \leq \alpha < \infty$, equation~\eqref{eq:bound_lim_zero} can be deduced from the fact that polynomials are dense in $A^p_\alpha$ (see e.g.~\cite{Kulikov2022}*{p.~939}). To extend~\eqref{eq:bound_lim_zero} to $0<\alpha<1$, we use Lemma~\ref{lem:prototype} which asserts that $A^p_\alpha$ is contained in, say, $A^{p/\alpha}_{1}=H^{p/\alpha}$. 
\end{proof}
\begin{proof}
	[Final part of the proof of Theorem~\ref{thm:HL}] In view of Lemma~\ref{lem:prototype}, all that remains is to show that equality in~\eqref{ineq:contractive_main} is attained only for the functions $f(z)=C\left(1-\overline{w}z\right)^{-2\alpha/p}$. Recall that Lemma~\ref{lem:alternative_norm} asserts that
	\[\|f_{\varrho}\|^{\sigma\alpha}_{\alpha,\sigma\alpha} = \Phi(\alpha,\sigma,f_\varrho) = t_0^{\alpha} - \alpha \int_0^{t_0} g_{\varrho}'(t) t^{\alpha - 1} \, dt \]
	holds for the functions $f_{\varrho}(z)=f(\varrho z)$ for almost every $0<\varrho<1$, where $t_0$ and $g_{\varrho}$ are defined as in Lemma~\ref{lem:kulikovg} for the function $f_{\varrho}$. By Lemma~\ref{lem:rho1}, we know that $\|f_{\varrho}\|^{\sigma\alpha}_{\alpha,\sigma\alpha} \to \|f\|^{\sigma\alpha}_{\alpha,\sigma\alpha}$ as $\varrho \to 1^-$. We now claim that we are done if we can show that 
	\begin{equation}\label{eq:Phi_finite} 
		\Phi(\alpha,\sigma,f) < \infty 
	\end{equation}
	for every function $f$ in $A^p_\alpha$. Indeed, if \eqref{eq:Phi_finite} holds, then we may repeat the argument from the proof of Lemma~\ref{lem:alternative_norm} to show that the function $\Phi$ and $\|f\|^{\sigma \gamma}_{\gamma,\sigma\gamma}$ are differentiable for $\mre{\gamma} > \alpha$ and since they are equal for real $\gamma > 1$ we get 
	\begin{equation}\label{eq:gamma_beta_F} 
		\Phi(\gamma, \sigma,f) = \|f\|^{\sigma \gamma}_{\gamma,\sigma\gamma}, 
	\end{equation}
	for all real $\gamma > \alpha$. 
	
	Fix any $\gamma_0$ in $(\alpha, \beta)$. Since
	\[\|f\|_{\beta,q} = \|f\|_{\beta, \sigma \beta} \leq \|f\|_{\gamma_0, \sigma \gamma_0} \leq \|f\|_{\alpha, \sigma \alpha} = \|f\|_{\alpha, p},\]
	we see that $\|f\|_{\alpha, p} = \|f\|_{\beta,q}$ implies that $\|f\|_{\beta, \sigma \beta} = \|f\|_{\gamma_0, \sigma \gamma_0}$. It follows from~\eqref{eq:gamma_beta_F} that $ {\Phi(\gamma_0,\sigma,f)}= \|f\|^{\sigma \gamma_0}_{\gamma_0, \sigma \gamma_0} $ and $\Phi(\beta,\sigma,f) = \|f\|^{\sigma \beta}_{\beta, \sigma \beta}.$ Corollary~\ref{cor:uniformzero} ensures that $f$ meets the assumptions of Lemma~\ref{lem:contr_bd} with parameters $\gamma_0$ and $\beta$, which implies
	\[t_0^\beta = \Phi(\beta, \sigma, f) = \|f\|^{\sigma\beta}_{\beta, \sigma\beta}. \]
	It follows from Corollary~\ref{cor:uniformzero} that $\max_{z\in \mathbb{D}} |f(z)|^{\kappa}(1-|z|^2)$ is attained at some point $w \in \mathbb{D}.$ Therefore, $$ |f(w)|^{q}(1-|w|^2)^{\beta} = \|f\|^q_{q, \beta}, $$ and from Corollary \ref{cor:pest} we deduce that $f(z) = C\left(1-\overline{w}z\right)^{-2\alpha/p}$ for a constant $C$ and a point $w$ in $\mathbb{D}$. 
	
	It remains to establish \eqref{eq:Phi_finite}. We observe that $g'(t)$ is well-defined and finite for every $t>0$ in view of Corollary~\ref{cor:uniformzero}, and our task is to show that
	\[-\int_0^{t_0} g'(t) t^{\alpha-1} \, dt < \infty.\]
	We resort again to a dyadic decomposition and the fact that $g'$ is nonpositive from Lemma~\ref{lem:kulikovg} (which is applicable due to Corollary~\ref{cor:uniformzero}) to see that
	\[-\int_0^{t_0} g'(t) t^{\alpha-1} \, dt \leq \sum_{n= 0}^\infty (g(t_0 2^{-n-1}) - g(t_0 2^{-n})) t_0^{\alpha-1}2^{-n(\alpha-1)}. \]
	Hence \eqref{eq:Phi_finite} will follow if we can show that 
	\begin{equation}\label{eq:indepN} 
		\sum_{n = 0}^N (g(t_02^{-n-1}) - g(t_02^{-n})) 2^{-n(\alpha-1)} \leq C 
	\end{equation}
	for a positive constant $C$ independent of $N$. By Lemma~\ref{lem:rho1}, we have $\|f_{\varrho}\|_{\alpha,p} \leq \|f\|_{\alpha,p}$ for all $\varrho$ in $(0,1)$. In addition, by Lemma~\ref{lem:alternative_norm}, we have
	\[\sum_{n = 0}^N (g_{\varrho}(t_02^{-n-1}) - g_{\varrho}(t_02^{-n})) 2^{-n(\alpha-1)} \leq C'\|f_{\varrho}\|^{p}_{\alpha,p}\]
	for almost all $0<\varrho<1$, where $C'$ does not depend on $N$. Therefore, the required estimate \eqref{eq:indepN} will follow if we can prove that for every fixed $t>0$ we have $g_\varrho(t)\to g(t)$. This is equivalent to $\mu_\varrho(t)\to \mu(t)$ for all $t > 0$ which holds by definition of $\mu$ and Corollary~\ref{cor:uniformzero} ensuring that all our sets are uniformly compactly embedded into the open unit disk for fixed $t > 0$. This finishes the proof of Theorem~\ref{thm:HL}. 
\end{proof}
\begin{remark}
	By carefully passing to the limit $\gamma\to \alpha^+$ we can in fact show that \eqref{eq:gamma_beta_F} holds for $\gamma = \alpha$ as well. 
\end{remark}
We wrap up the present section with the proof of Corollary~\ref{cor:majorant}.
\begin{proof}
	[Proof of Corollary~\ref{cor:majorant}] Part (a) of Corollary~\ref{cor:majorant} follows from Theorem~\ref{thm:HL} if we first use the Cauchy--Schwarz inequality, the binomial series \eqref{eq:binomialseries}, and Lemma~\ref{thm:hilbertspace} to the effect that 
	\begin{multline*}
		\sum_{k=0}^\infty |a_k| r^k \leq \left(\sum_{k=0}^\infty c_{2\alpha/p}(k) r^{2k}\right)^{\frac{1}{2}} \left(\sum_{k=0}^\infty \frac{|a_k|^2}{c_{2\alpha/p}(k)}\right)^{\frac{1}{2}} \\= (1-r^2)^{-\alpha/p} \|f\|_{A^2_{2\alpha/p}}\leq (1-r^2)^{-\alpha/p}\|f\|_{A^p_\alpha}. 
	\end{multline*}
	
	To establish Part (b) of Corollary~\ref{cor:majorant}, we fix $0<r<1$ and consider the function
	\[f_r(z) = \left(1-r z\right)^{-2\alpha/p}-2.\]
	We need to prove that 
	\begin{equation}\label{eq:suff} 
		(Mf_r(r))^p (1-r^2)^{\alpha} > \|f_r\|_{\alpha,p}^p. 
	\end{equation}
	Note that $f_r$ coincides with $\left(1-r z\right)^{-2\alpha/p}$ except that we have changed the sign of the constant term in its Taylor series. This implies in particular that 
	\begin{equation}\label{eq:Mbound} 
		Mf_r(r) = \frac{1}{(1-r^2)^{2\alpha/p}}. 
	\end{equation}
	However, using the second assertion of Theorem~\ref{thm:HL}, we find that
	\[\|f\|_{\alpha,p}^p < \|f\|_{A^2_{2\alpha/p}}^p = \left(\sum_{k=0}^\infty \frac{|a_k|^2}{c_{2\alpha/p}(k)}\right)^{\frac{p}{2}} = \frac{1}{(1-r^2)^{\alpha}},\]
	when $2<p<\infty$. This implies \eqref{eq:suff} in view of \eqref{eq:Mbound}. 
\end{proof}

\section{Comparison with Besov spaces} \label{sec:besov} 
We split the proof of Theorem~\ref{thm:besovcomp} into four parts. The first two parts are the inclusion in (a), where different arguments are used to handle the ranges $0 < p \leq 1$ and $1 < p \leq 2$. The third part is the inclusion from (b) and the final part of the proof is the assertion that $A^p_\alpha \neq B^p_\alpha$ for $p\neq2$.
\begin{proof}
	[Proof of Theorem~\ref{thm:besovcomp}~(a) for $0 < p \leq 1$] We note that the inequality
	\[\frac{d}{dr}|f(re^{i\theta})|^{p} \leq p |f(r e^{i\theta})|^{p-1} |f'(re^{i\theta})| \]
	yields the bound
	\[\int_{\frac12}^r \frac{d}{dr}M^p_p(r,f) (1-r^2)^{\alpha-1} dr \leq 2p \int_{r\mathbb{D}} |f(z)|^{p-1} |f'(z)| (1-|z|^2)^{\alpha-1} dm(z). \]
	We are done when $p=1$ by passing to the limit $r\to 1^-$. For $0<p<1$, we use H\"{o}lder's inequality to get 
	\begin{multline*}
		\int_{r\mathbb{D}} |f(z)|^{p-1} |f'(z)| (1-|z|^2)^{\alpha-1} \, dm(z) \\
		\leq \left(\int_{r\mathbb{D}} |f'(z)|^{p} \left|\frac{f'(z)}{f(z)} \right|^{2-p} (1-|z|^2)^{\alpha}\, dm(z)\right)^{\frac{1-p}{2-p}} \\
		\times \left(\int_{r\mathbb{D}} |f'(z)|^p(1-|z|^2)^{\alpha-2+p} \, dm(z) \right)^{\frac{1}{2-p}}. 
	\end{multline*}
	Now employing Lemma~\ref{lem:LP} and passing to the limit $r\to 1^-$, we get the desired bound
	\[ \| f\|_{\alpha,p}^p \leq C_{\alpha,p} \| f \|_{B_{\alpha}^p}. \qedhere\]
\end{proof}

For the next part of the proof we require two preliminary results. 
\begin{lemma}\label{lem:Lp0} 
	For $1\leq p\leq 2$, there exists a constant $C_p$ such that for all $f$ in $L^p([0, 1])$ we have
	\[ \left\| f \right\|_p^p \leq \left|\int_0^1 f(x)\, dx\right|^p+C_p \left\| f-\int_0^1f(x)\,dx \right\|_p^p.\]
\end{lemma}
\begin{proof}
	We show first that for $1\leq p\leq 2$, there exists a positive constant $C_p$ such that 
	\begin{equation}\label{eq:allz} 
		|1+z|^p \leq 1 + p\mre{z} + C_p |z|^p 
	\end{equation}
	for all complex numbers $z$. For $|z|\leq \frac{1}{2}$, we use the linear approximations of $(1+z)^{\frac{p}{2}}$ and $(1+\overline{z})^{\frac{p}{2}}$ to see that $|1+z|^p - 1 - p\mre{z}$ is bounded by a constant times $|z|^2$. This yields \eqref{eq:allz} since $|z|^2\leq |z|^p$ when $|z|\leq 1$ and $p\leq 2$. For $|z|>\frac{1}{2}$, the left-hand side of \eqref{eq:allz} is bounded from above by $3^p|z|^p$ and the right-hand side is bounded from below by $(C_p-p2^{p-1})|z|^p$, and so \eqref{eq:allz} holds in this range as well if we choose $C_p\geq p2^{p-1}+3^p$.
	
	The lemma is trivially true when $\int_0^1 f(x)\,dx=0$, so we may assume that $\int_0^1f(x)\,dx=1$ by scaling. Applying \eqref{eq:allz} with $z=f-1$ and integrating over $[0,1]$, we get
	\[\|f\|_p^p \leq 1+ C_p \| f-1 \|_p^p ,\]
	since the integral of $\mre f -1 $ clearly vanishes. 
\end{proof}

Lemma~\ref{lem:Lp0} holds plainly with $C_1=1$ by the triangle inequality and with $C_2=1$ by orthogonality. In the latter case, the inequality is in fact an equality. Numerical examples suggest that $C_p>1$ in the range $1<p<2$.

We will also need the following characterization of Besov spaces which is a special case of a theorem of Dyakonov~\cite{Dyakonov98}*{Theorem~2.1}.
\begin{lemma}\label{lem:dyak} 
	Fix $0<\alpha<1$ and $1\leq p<\infty$. A function $f$ in $H^p$ is in $B_\alpha^p$ if and only if
	\[\int_0^1 \int_0^{2\pi} \int_0^{2\pi} |f(e^{it})-f(re^{i\theta})|^p \frac{(1-r^2)^{\alpha-1}}{|e^{it}-re^{i\theta}|^2} \frac{dt}{2\pi} \frac{d\theta}{2\pi} dr<\infty .\]
\end{lemma}

We are now ready to continue with the second part of the proof of Theorem~\ref{thm:besovcomp}~(a).
\begin{proof}
	[Proof of Theorem~\ref{thm:besovcomp}~(a) for $1<p \leq 2$] Since $r\mapsto M_p(r,f)$ is an increasing function, we get the bound 
	\begin{multline*}
		\int_0^1 \left(\frac{d}{dr}M^p_p(r,f)\right) (1-r^2)^{\alpha-1} \,dr \\
		\leq C \int_0^1 \left(M_p^p(1,f)-M_p^p(r,f)\right) (1-r^2)^{\alpha-2} \,dr, 
	\end{multline*}
	with $C$ a constant depending on $\alpha$ and $p$. Here we again use the fact $f$ is in $H^p$, and so we declare that $M(1,f)\coloneqq \|f\|_{1,p}^p$. By Lemma~\ref{lem:dyak}, the proof will be complete if we can show that
	\[ M_p^p(1,f)-M_p^p(r,f) \leq C_p \int_{0}^{2\pi} \int_{0}^{2\pi} |f(e^{it})-f(re^{i\theta})|^p \frac{(1-r^2)}{|e^{it}-re^{i\theta}|^2} \frac{dt}{2\pi} \frac{d\theta}{2\pi}\]
	for $1<p\leq 2$. By Fubini's theorem, we may write
	\[ M_p^p(1,f)-M_p^p(r,f) = \int_{0}^{2\pi}\int_{0}^{2\pi} \left(|f(e^{it})|^p - |f(re^{i\theta})|^p\right) \frac{(1-r^2)}{|e^{it}-re^{i\theta}|^2} \frac{dt}{2\pi}\frac{d\theta}{2\pi} ,\]
	and so we are done if we can get 
	\begin{multline*}
		\int_{0}^{2\pi} |f(e^{it})|^p \frac{(1-r^2)}{|e^{it}-re^{i\theta}|^2} \frac{dt}{2\pi} - |f(re^{i\theta})|^p \\
		\leq C_p \int_{0}^{2\pi} |f(e^{it})-f(re^{i\theta})|^p \frac{(1-r^2)}{|e^{it}-re^{i\theta}|^2} \frac{dt}{2\pi} 
	\end{multline*}
	uniformly in $r$ and $\theta$. By the change of variables
	\[ e^{i\tau} \mapsto \frac{re^{i \theta}-e^{it}}{1-re^{-i\theta+it}} \]
	in the integrals, this can be simplified to the inequality
	\[ \| f \|_p^p-| f(0) |^p \leq C_p \| f-f(0) \|_p^p, \]
	which holds for all functions $f$ in $H^p$ in view of Lemma~\ref{lem:Lp0}. 
\end{proof}
Part (b) of Theorem~\ref{thm:besovcomp} can be proved in essentially the same way as done by Luecking~\cite{Luecking88} in the classical case $\alpha=1$. 
\begin{proof}
	[Proof of Theorem~\ref{thm:besovcomp}~(b)] We start from Luecking's inequality
	\[|f'(0)|^p \leq C_p \int_{|z|<\frac12} |f(z)|^{p-2} |f'(z)|^2 \log{\frac{1}{2|z|}} \,dm(z). \]
	Setting
	\[\psi_a(z)\coloneqq \frac{a-z}{1-\overline{a}z} \]
	and applying this inequality to $z\mapsto f(\psi_a(z))$, we find that 
	\begin{multline*}
		|f'(a)|^p (1-|a|^2)^p \\
		\leq C_p \int_{|z|<\frac12} |f(\psi_a(z))|^{p-2} |f'(\psi_a(z))|^2 |\psi'_a(z)|^2 \log{\frac{1}{2|z|}} \,dm(z). 
	\end{multline*}
	Now integrating this inequality with respect to $(1-|a|^2)^{\alpha-2}dm(a)$ over $\mathbb{D}$, we get the desired quantity on the left-hand side. On the right-hand side, we follow Luecking, and so we apply Fubini's theorem and make the change of variable $ a\mapsto w=\psi_a(z) $ in the integral. The only difference from the proof in \cite{Luecking88} is that we get an additional factor $(1-|w|^2)^{\alpha}$ in the integral on the right-hand side. To achieve this, we use that
	\[|w|^2 = |\psi_a(z)|^2 = 1-\frac{(1-|a|^2)(1-|z|^2)}{|1-\overline{a}z|^2}\]
	so that $1-|a|^2\asymp 1-|w|^2$ since $|z|<1/2$. Thus, the integral on the right-hand side is bounded by a constant times
	\[ \int_{\mathbb{D}} |f(w)|^{p-2} |f'(w)|^2 (1-|w|^2)^{\alpha} \,dm(w) ,\]
	as required. 
\end{proof}

In preparation for the final part of the proof of Theorem~\ref{thm:besovcomp}, we set
\[\varrho(z,w)\coloneqq \left|\frac{z-w}{1-\overline{z}w}\right|,\]
which is the pseudohyperbolic distance between $z$ and $w$ in $\mathbb{D}$. We say that a sequence $Z=(z_j)_{j\geq1}$ is uniformly discrete if $\inf_{j\neq k}\varrho(z_j,z_k)>0$. The following result is a consequence of \cite{Seip95}*{Theorem~2}.
\begin{lemma}\label{lem:sine} 
	Fix $\gamma>0$. Then there exists an analytic function $g$ on $\mathbb{D}$ whose zero set $Z$ is uniformly discrete and which satisfies
	\[|g(z)| \asymp \varrho(z,Z) (1-|z|^2)^{-\gamma}\]
	for $z$ in $\mathbb{D}$. 
\end{lemma}

All that remains in the proof of Theorem~\ref{thm:besovcomp} is to prove that $A^p_\alpha \neq B^p_\alpha$ when $0 < \alpha \leq 1$. The fact that $H^p\neq B_{1}^p$ is well known, as pointed out in the introduction. We are therefore left with the following.
\begin{proof}
	[Proof that $A^p_\alpha \neq B^p_\alpha$ for $0<\alpha<1$ and $p\neq2$] When $0<p<2$, we invoke Lemma~\ref{lem:sine} with $\gamma$ satisfying
	\[\frac{\alpha+p-1}{p} < \gamma < \frac{\alpha+1}{2}.\]
	Since $0<\gamma<1$ (due to our standing assumption that $\alpha+p>1$ for Besov spaces), there is then a function $f$ in $H^\infty$ such that $f'=g$. It is clear that
	\[\int_{\mathbb{D}} |f'(z)|^p (1-|z|^2)^{p-1} \,dm(z) =\infty,\]
	so $f$ is not in $B_\alpha^p$. On the other hand, by adding to $f$ a suitable constant, we may assume that also $1/f$ is in $H^\infty$. Then Lemma~\ref{lem:LP} shows that $f$ belongs to $A_{\alpha}^p$.
	
	We act similarly when $2<p<\infty$, the task being to identify a function $f$ not belonging to $A_{\alpha}^p$ such that 
	\begin{equation}\label{eq:besov} 
		\int_{\mathbb{D}} |f'(z)|^p (1-|z|^2)^{\alpha +p-2} \,dm(z) <\infty. 
	\end{equation}
	We choose $f$ as in the preceding case but now with
	\[\frac{\alpha+1}{2} < \gamma < \frac{\alpha+p-1}{p}.\]
	It is then immediate that \eqref{eq:besov} holds. By again adding a constant to $f$, we may ensure that both $f$ and $1/f$ are bounded. The fact that $f$ is not in $A_{\alpha}^p$ then follows from Lemma~\ref{lem:LP}. 
\end{proof}

\section{The shift operator and division by inner functions} \label{sec:shift} 
The starting point for the proof of both Theorem~\ref{thm:exp} and Theorem~\ref{thm:inner} is the formula 
\begin{equation}\label{eq:fgdiff} 
	\begin{split}
		\|f\|^p_{\alpha,p} - \|g\|^p_{\alpha,p} = &\int_0^1 \left(\frac{d}{dr} \left( M_p^p(r,f)-M_p^p(r,g)\right)\right)(1-r^2)^{\alpha-1}\, dr \\
		&\qquad\qquad\qquad\qquad\qquad\qquad\quad\,\,\,\,\,+ |f(0)|^p-|g(0)|^p. 
	\end{split}
\end{equation}
The basic idea is that cancellation between $M_p^p(r,f)$ and $M_p^p(r,g)$ as $r \to 1^-$ allows us to integrate by parts.
\begin{proof}
	[Proof of Theorem~\ref{thm:exp}] Since $M_p^p(r,Sf) = r^p M_p^p(r,f)$ and $Sf(0)=0$, the formula \eqref{eq:fgdiff} takes the form
	\[\|Sf\|^p_{\alpha,p} - \|f\|^p_{\alpha,p} = - \int_0^1 \frac{d}{dr} \left(M_p^p(r,f)(1-r^p)\right) (1-r^2)^{\alpha-1}\,dr - |f(0)|^p.\]
	Using that $f$ is in $A^p_1=H^p$ (since $\|f\|_{1,p} \leq \|f\|_{\alpha,p}$ for $0<\alpha<1$) and that
	\[\lim_{r\to 1^-} (1-r^p)(1-r^2)^{\alpha-1} = 0\]
	for $\alpha>0$, we can integrate by parts and get 
	\begin{equation}\label{eq:Sfest} 
		\|Sf\|_{\alpha,p}^p - \|f\|_{\alpha,p}^p = 2 (1-\alpha) \int_0^1 M_p^p(r,f) (1-r^p)(1-r^2)^{\alpha-2}\,rdr. 
	\end{equation}
	The right-hand side of \eqref{eq:Sfest} is positive when $f \not \equiv 0$, which demonstrates that $S$ is a strict expansion. Since $M_p^p(r,f) \leq \|f\|_{1,p}^p \leq \|f\|_{\alpha,p}^p$ (the second inequality follows from \eqref{eq:Apa}), we can also infer from \eqref{eq:Sfest} that
	\[\|Sf\|_{\alpha,p}^p \leq \left(1 + 2 (1-\alpha) \int_0^1 (1-r^p)(1-r^2)^{\alpha-2}\,rdr\right)\|f\|_{\alpha,p}^p.\]
	We complete the proof by noting that $M_p^p(r,f) = \|f\|_{1,p}^p = \|f\|_{\alpha,p}^p$ holds if and only if $f$ is a constant function. 
\end{proof}
\begin{proof}
	[Proof of Theorem~\ref{thm:inner}] We now use \eqref{eq:fgdiff} with $g \coloneq f/I$. If we can prove that 
	\begin{equation}\label{eq:rto1} 
		M_p^p(r,g)-M_p^p(r,f)=o\left((1-r^2)^{1-\alpha}\right), 
	\end{equation}
	then integration by parts yields
	\[\|f\|_{\alpha,p}^p - \|g\|_{\alpha,p}^p = 2(\alpha-1) \int_0^1 \left(M_p^p(r,g)-M_p^p(r,f)\right)(1-r)^{\alpha-2}\,rdr.\]
	Here the right-hand side is positive since we assume that $I$ is a nontrivial inner function and $f$ is nontrivial. Hence it remains only to establish \eqref{eq:rto1}. Since $M_p^p(r,f)\leq M_p^p(r,g)\leq M_p^p(1,f)$, it suffices to show that 
	\begin{equation}\label{eq:nondy} 
		M_p^p(1,f)-M_p^p(r,f)=o\left((1-r^2)^{1-\alpha}\right). 
	\end{equation}
	This follows if we use the dyadic decomposition 
	\begin{multline*}
		\sum_{n=0}^\infty \left(M_p^p(1-2^{-n-1},f) - M_p^p(1-2^{-n},f)\right) 2^{(1-\alpha)(n+1)} \\
		\leq 2 \int_0^1 \left(\frac{d}{dr} M_p^p(r,f)\right) (1-r^2)^{\alpha-1} \,dr. 
	\end{multline*}
	Indeed, this bound yields
	\[M_p^p(1-2^{-n-1},f) - M_p^p(1-2^{-n},f) = o(2^{-(1-\alpha)n}).\]
	Hence by summation, we then get
	\[M_p^p(1,f) - M_p^p(1-2^{-n},f) = o(2^{-(1-\alpha)n}),\]
	which in turn implies \eqref{eq:nondy} by monotonicity of $M_p^p(r,f)$ in $r$. 
\end{proof}

\section*{Acknowledgements} We are grateful to Jos\'{e} \'{A}ngel Pel\'{a}ez for enlightening comments on the history of the classical case $\alpha=1$ of Theorem~\ref{thm:besovcomp}. 

Part of this project was carried out while three of the authors (Aleksei Kulikov, Kristian Seip, Ilya Zlotnikov) were participating in the Intensive Research Programme on Modern Trends in Fourier Analysis at Recerca de Matem\`{a}tica in Barcelona during May--June 2025. They would like to thank the host and organizers for their hospitality.

\bibliography{bohrhp}

\end{document}